\documentclass[11pt,a4paper,oneside]{article}
\usepackage[tight,hang]{subfigure}
\usepackage{amsmath,amssymb,amsfonts,amsthm}
\usepackage{color}
\usepackage[tight,hang]{subfigure}
\usepackage{amsmath,color,graphicx}
\usepackage{amsmath,color,graphicx}
\usepackage[english]{babel}
\usepackage{hyperref}
\usepackage{etoolbox}
\usepackage{graphicx}
\usepackage{epstopdf}
\usepackage{subfloat}

\usepackage[tight,hang]{subfigure}
\usepackage{amsmath,color,graphicx}

\newcommand\overcirc[1]{\raisebox{10pt}{\tiny$\circ$}{\kern-6.5pt}\mbox{$#1$}}
\newcommand\undersym[2]{\raisebox{-5pt}{\tiny$#2$}{\kern-12pt}\mbox{$#1$}}
\newcommand\oversym[2]{\raisebox{8pt}{\tiny$#2$}{\kern-10pt}\mbox{$#1$}}

\theoremstyle{plain}
\newtheorem{thm}{Theorem}[section]

\newtheorem{defn}[thm]{Definition}
\newtheorem{example}{Example}

\begin{document}	
	\date{}
	\title{\bf{Efficient method for fractional L\'{e}vy-Feller advection-dispersion equation using Jacobi polynomials}}
	\maketitle
	\vspace{-1.13cm}
	\vspace{-0.5cm}
	\begin{center}
		\textbf{N. H. Sweilam$^{a}$, M. M. Abou Hasan$^{b}$}\end{center}
	\begin{center}
		\small{$^{a,b}$Department of Mathematics, Faculty of Science, Cairo University, Giza, Egypt}\\

		\small{E-mails:  nsweilam@sci.cu.edu.eg $^{a}$, muneere@live.com $^{b}$	}\\
	\end{center}
\textbf{Abstract:} \\

  In this paper, a novel formula expressing explicitly the fractional-order derivatives, in the sense of Riesz-Feller operator, of Jacobi polynomials is presented. Jacobi spectral collocation method together with trapezoidal rule are used to reduce the fractional L\'{e}vy-Feller advection-dispersion equation (LFADE) to a system of algebraic equations which greatly simplifies solving like this fractional differential equation.  Numerical simulations with some comparisons are introduced to confirm the effectiveness and reliability of the proposed technique for the L\'{e}vy-Feller fractional partial differential equations. 
\begin{flushleft}

\end{flushleft}
%

\section{Introduction}
\qquad Recently, the notion of fractional calculus and differential operators of
fractional order has gained great popularity due to its engaging applications in a different fields, such as engineering, finance, system control, hydrology, viscoelasticity, and physics (\cite{Benson D A}, \cite{Podlubny}, \cite{Raberto M}, \cite{Koeller R C}, \cite{R. Metzler}). The fractional models often can be described as an ordinary fractional differential equations or partial fractional differential equations. These differential equations are more appropriate than the standard integer-order ones for describing the memorial and genetic characteristic of many phenomena and materials \cite{Samko S G}. 

As is well-known, analytic solutions of most fractional differential equations can not always be obtained explicitly, and therefore the employment of different numerical techniques for solving such equations is necessary. Some of the proposed numerical methods for solving such equations are: finite difference methods (\cite{F. Zeng}, \cite{H. Zhang}, \cite{M.Ci4}, \cite{M.Ci42}, \cite{N. H. hnansp} \cite{N. H. MunerNONL}), finite element methods \cite{V.J. Ervin}, semi-analytic methods (\cite{A. Golbabai},  \cite{V. Daftardar},  \cite{Y. Cenesiz},  \cite{S. Das}), spectral methods (\cite{Bhrawy2zaky}, \cite{N. H. Sspectr2}, \cite{N. H. Muner}), higher order numerical methods \cite{Y. Yan}.

Spectral methods are a class of techniques in which the numerical solutions are expressed in
terms of certain ”basis functions”, which may be orthogonal polynomials (see, e.g.  \cite{David}, \cite{Jie Shen}, \cite{Trefethen} and the references therein). Collocation method, one of the three well-known kinds of spectral methods, has become increasingly common for solving ordinary and partial differential equations. Spectral collocation method has an exponential convergence rate, so it is valuable in providing highly accurate solutions to nonlinear differential equations even using a small number of grids. Choosing of collocation points has very important role in the convergence and efficiency of the collocation method \cite{Jie Shen}. In  reality, spectral collocation methods have been used for the linear and nonlinear fractional partial differential equations (\cite{Bhrawy2zaky}, \cite{Khader2}, \cite{N. H. Sspectr2}, \cite{N. H. Sspectr3}, \cite{Shahrokh}) and for the fractional integro-differential equations (\cite{Khader}, \cite{Eslahchi}, \cite{N. H. Sspectr1}, \cite{MaX}).

The Jacobi polynomials, which we will denote by ($J_k^{\beta,\gamma}(x) ,\  k\geqslant 0, \  \beta>-1, \gamma>-1$), have been used extensively in mathematical analysis and practical applications, and play an important role in the analysis and implementation of spectral methods. The use of general Jacobi polynomials has the advantages of obtaining the solutions in terms of the Jacobi parameters $\beta$ and $\gamma$. Hence to generalize and instead of developing approximation results for each particular pair of indexes, it would
be very useful to carry out a systematic study on Jacobi polynomials with general indexes, which can then be
directly applied to other applications. 

In physics, anomalous diffusion phenomena are modeled using fractional derivatives, such that the particles sawing differently than the classical Brownian motion model \cite{R. Metzler}, they follow L\'{e}vy stable motion \cite{M.Ci4}. The Fokker-Planck equations, Reaction-diffusion equations, diffusion advection equations, and Kinetic equations of the diffusion can be applications of the phenomenon of anomalous diffusion. The fractional advection-dispersion equation (FADE), also known as the fractional kinetic equation \cite{R. Hilfer}, was shown to be an extension of the continuous-time random walk model. In groundwater hydrology FADE is utilized to model the transport of passive tracers carried by fluid flow in a porous medium \cite{Q. Liu}.

Considerable numerical methods for solving FADE have been proposed. Meerschaert \textit{et al.} \cite{M.M. Meerschaert} developed practical numerical methods to solve the one-dimensional space FADE with variable coefficients on a finite domain. Roop \cite{J. P. Roop} proposed the finite-element method to approximate FADE in two spatial dimensions. Liu \textit{et al.} \cite{Q. Liu} proposed an approximation of the L\'{e}vy-Feller advection-dispersion process by employing a random walk and finite difference methods. EI-Sayed \textit{et al.} \cite{A. M. A. EI-Sayed} used Adomian's decomposition method to solve an advection-dispersion model with a fractional derivative in the Caputo sense. Golbabai \textit{et al.} \cite{A. Golbabai} used homotopy perturbation method for finding analatic solution of FADE. Shen \textit{et al.} \cite{S. Shen} used weighted Riesz fractional finite-difference approximation scheme for FADE. Recently, Bhrawy \textit{et al.} presented the operational matrices for solving the two sided FADE in \cite{Bhrawy4}. More recently, Feng \textit{et al.} \cite{L.B. Feng} based on the weighted and shifted Gr\"{u}nwald difference (WSGD) operators they approximated the Riesz space fractional derivative to find numerical solution of Riesz space FADE. 


In this paper, we consider the L\'{e}vy-Feller advection-dispersion equation (LFADE) \cite{Q. Liu}, with source term, which obtains form standard advection-dispersion equation by replacing the second-order space derivative with the Riesz-Feller derivative of order $\alpha$ and skewness $\theta$, ($|\theta|  \leq min \{\alpha, 2 - \alpha\}$). LFADE takes the following form:


\begin{equation}\label{LVAd-RF}
\frac{\partial u(x,t)}{\partial t}=dD^{\alpha}_{\theta}u(x,t)-e\frac{\partial u(x,t)}{\partial x}+s(x,t), \ \ \ \ d>0, \ \ e\geq 0,\ \  t\geq0, \ \ \ x \in \mathbb{R},
\end{equation}
with initial condition:  \begin{equation} \label{init levyfeller ad}
u(x,0)=f(x),
\end{equation}
where the operator $D^{\alpha}_{\theta}$ is the Riesz-Feller fractional derivative of order $\alpha$ and skewness $\theta$.  
The constants $d,\ e$ represent the dispersion coefficient and the average fluid velocity respectively 
and $s(x,t)$ is the source term.
The fundamental solution of (\ref{LVAd-RF}) and (\ref{init levyfeller ad}) has been derived using the Fourier transform \cite{F. Huang} as:
\begin{equation}\label{FUsOLVAD}
{{G'}}_{\alpha}(k,t;\theta)=exp(-ta|k|^{\alpha}e^{isign(k)\theta\pi/2}+itbk), \ \ \ k\in \mathbb{R}.
\end{equation}

Numerical studies of Eq.(\ref{LVAd-RF}) have been mostly obtained by finite difference methods (FDMs) (see \cite{Q. Liu})  with their limited accuracy. That because FDMs have a local character, while fractional derivatives are essentially global differential operators. Hence, global schemes such as spectral methods may be more appropriate for discretizing fractional operators.


Our purpose of this paper is to construct an accurate numerical technique to solve (\ref{LVAd-RF}) and (\ref{init levyfeller ad}) using Jacobi spectral collocation (JSC) method combined with the trapezoidal rule (Crank-Nicolson method) in the one dimensional domain $ \Omega: a\leq x \leq b$ subject to the Dirichlet boundary conditions as: \begin{equation} \label{bound levyfeller ad}
 u(a,t)=0, \ u(b,t)=0 ,
\end{equation} 
More precisely, implementing JSC method to the spatial variable of the fractional advection-dispersion
equation and using the boundary conditions reduces the problem to solving a
system of ordinary differential equations with respect to the time variable. Then this system will be solved using the trapezoidal rule  to reduce the problem to solve system of algebraic equations which are far easier to be solved. This is a generalization of the previous authors' work in \cite{N. H. Muner}. 

Indeed, there is a little work in the literature for solving numerically fractional differential equations when the Riesz-Feller operator is used to describe the fractional derivatives (see \cite{H. Zhang}, \cite{M.Ci4}, \cite{M.Ci42}, \cite{N. H. Tuan} and \cite{Q. Liu}, most of them used FDM). We used in the previous work \cite{N. H. Muner} Chebyshev-Legendre collocation method with first-order Euler method for solving L\'{e}vy-Feller diffusion equation. To the best of our knowledge, there is no paper used spectral method for solving L\'{e}vy-Feller advection-dispersion equation, and this motivated our interest in such method.

This article is organized as follows: In the following section, we will write down some definitions on fractional calculus and give some relevant properties of Jacobi polynomials. In Section 3, we suggest and prove an explicit formula corresponding to Riesz-Feller fractional derivative of Jacobi polynomials. In Section 4 we applied JSC method to solve (\ref{LVAd-RF}, \ref{init levyfeller ad}, \ref{bound levyfeller ad}) on $ \Omega$ and change them to a system of ordinary differential equations which will be solved using the trapezoidal rule. Section 5, reports some numerical results to show the accuracy and the applicability of the proposed method. Finally, in Section 6 some conclusions are given.

\section{Definitions and fundamentals}
\qquad Here, we introduce some necessary definitions and mathematical preliminaries of the fractional derivative theory.
\subsection{Some properties of fractional calculus}

\begin{defn}\label{difofrf}  
	
\ \ For $0 < \alpha < 2$, $\alpha \neq1$ and $|\theta|  \leq min \{\alpha, 2 - \alpha\}$, the Riesz-Feller fractional operator $D^{\alpha}_{\theta}$ represents in the following form (see e.g.  \cite{B.Al}, \cite{H. Zhang}, \cite{M.Ci4}, \cite{M.Ci42})

	\begin{equation}\label{R-F-2}
	D^{\alpha}_{\theta}f(x)=-(c_+D^{\alpha}_{+}+c_-D^{\alpha}_-)f(x),
	\end{equation}
	where the coefficients $c_\pm$ are given by
	\begin{equation}\label{c+-}
	c_+=c_+(\alpha, \theta)=\frac{sin((\alpha-\theta)\pi/2)}{sin(\alpha\pi)},\ \
	c_-=c_-(\alpha, \theta)=\frac{sin((\alpha+\theta)\pi/2)}{sin(\alpha\pi)},
	\end{equation}
	and
	\begin{equation}\label{D+-}
	(D^{\alpha}_{+}f)(x)=(\frac{d}{dx})^n(I^{n-\alpha}_{+}f)(x),\ \
	(D^{\alpha}_{-}f)(x)=(-\frac{d}{dx})^n(I^{n-\alpha}_{-}f)(x),
	\end{equation}
	are the left-side and right-side Riemann-Liouville fractional derivatives with $x\in\mathbb{R}$ and $\alpha>0, \ n-1<\alpha\leq n, \ n=1,\ 2$. 
\end{defn}
	In expressions (\ref{D+-}) the fractional operators $I^{n-\alpha}_{\pm}$ are defined as the left-and right-side of Weyl fractional integrals, which given by the following definition: 
	\begin{defn} \ For $\alpha>0,$
	\begin{equation}\label{wely}
	(I^{\alpha}_{+}f)(x)=\frac{1}{\Gamma(\alpha)}\int^{x}_{-\infty}\frac{f(\xi)}{(x-\xi)^{1-\alpha}}d\xi,\ \
	(I^{\alpha}_{-}f)(x)=\frac{1}{\Gamma(\alpha)}\int^{+\infty}_{x}\frac{f(\xi)}{(\xi-x)^{1-\alpha}}d\xi.
	\end{equation}
\end{defn}
	For $\alpha=1$, the representation (\ref{R-F-2}) is not valid and has to be replaced by the formula
	\begin{equation}
	D^{1}_{\theta}f(x)=[cos(\theta \pi /2)D^{1}_{0}-sin(\theta \pi /2)D]f(x),
	\end{equation}
	where the operator $D_0^1 $ is related to the Hilbert transform as first noted by Feller in 1952 in his pioneering paper \cite{feller}
	$$ D_0^1=\frac{1}{\pi}\frac{d}{dx}\int _{-\infty}^{+\infty}\frac{f(\xi)}{x-\xi}d\xi,$$ and $D$ refers for the first standard derivative.\\
	
   \hspace{-.9cm} For $\alpha=2$\  (surely, $\theta=0$), 	$D^{\alpha}_{\theta}f(x)=\frac{d^2f(x)}{dx^2}.$

From definition (\ref{difofrf}), we see the Riesz-Feller fractional derivative is a linear combination of left-side and right-side Riemann-Liouville fractional derivatives, so: $$	D^{\alpha}_{\theta}(\lambda f(x) +\gamma g(x))=\lambda D^{\alpha}_{\theta} f(x) +\gamma D^{\alpha}_{\theta}g(x).$$
We also recall from \cite{Podlubny} a useful property of the left-side and right-side Riemann-Liouville fractional derivatives. Assume that $x\in [a,b], \ a,b \in \mathbb{R}$ such that $f(a)=f(b)=0$ then for $0<p,q\leq 1$ we have 
\begin{equation}\label{int-derv}
_aD^{p}_{x}\ _aI^{p}_xf(x)=f(x),\qquad _xD^{p}_{b} \ _xI^p_bf(x)=f(x).
\end{equation}
Also, for Riesz-Feller fractional derivative we have:

\subsection{Some properties of Jacobi polynomials}
We introduce in this section some basic properties of (shifted) Jacobi polynomials ($J_k^{\beta,\gamma}(x) ,\ \  k\geqslant 0,\   \beta>-1, \ \ \gamma>-1$) that are most relevant to the proposed spectral collocation approximations (\cite{T.S. Chihara}, \cite{R. Koekoek}, \cite{M. E. H. Ismail}, \cite{A. F. Nikiforov}).

Jacobi polynomials may be archived from the recurrence relation:
$$J_{k+1}^{\beta,\gamma}(x)=(a_k^{\beta,\gamma }x-b_k^{\beta,\gamma})J_k^{\beta,\gamma}(x)-c_k^{\beta
	,\gamma}J_{k-1}^{\beta,\gamma}(x), k=1, \ 2,\ ...,$$
 $$ J_0^{\beta,\gamma}(x)=1,\qquad J_1^{\beta,\gamma}(x)=\frac{(\beta+\gamma+2)x+\beta-\gamma}{2},$$
  where $$ a_k^{\beta,\gamma}=\frac{(2k+\beta+\gamma+1)(2k+\beta+\gamma+2)}{2(k+1)(k+\beta+\gamma+1)},$$
    $$ b_k^{\beta,\gamma}=\frac{(2k+\beta+\gamma+1)(\gamma^2-\beta^2)}{2(k+1)(k+\beta+\gamma+1)(2k+\beta+\gamma)},$$
        $$ c_k^{\beta,\gamma}=\frac{(k+\beta)(k+\gamma)(2k+\beta+\gamma+2)}{(k+1)(k+\beta+\gamma+1)(2k+\beta+\gamma)}.$$ 
 We mention here some of the most important properties of Jacobi polynomials       
  \begin{equation}\label{pro-jacobi-1}
J_k^{\beta,\gamma}(-x)=(-1)^kJ_k^{\gamma,\beta}(x),\qquad J_k^{\beta,\gamma}(1)=\frac{\Gamma(k+\beta+1)}{k!\Gamma(\beta+1)},
  \end{equation}
 \begin{equation}\label{pro-jacobi-2}\frac{d^m}{dx^m}J_k^{\beta,\gamma}(x)=\frac{\Gamma(m+k+\beta+\gamma+1)}{2^m\Gamma(k+\beta+\gamma+1)}J_{k-m}^{\beta+m,\gamma+m}(x).\end{equation}
  A basic property of the Jacobi polynomials is that they are the eigenfunctions of the
  singular Sturm-Liouville problem:
  $$(1-x^2)\phi''(x)+[\gamma
  -\beta+(\beta+\gamma+2)x]\phi'(x)+k(k+\beta+\gamma+1)\phi(x) = 0.$$

In order to use these polynomials on the interval $[0, L],\ L>0,$ we recall here the shifted Jacobi polynomials  $J_{L,k}^{\beta,\gamma}(x)=J_k^{\beta,\gamma}(\frac{2x-L}{L}).$\\
The analytic form of the shifted Jacobi polynomials $J_{L,k}^{\beta,\gamma}(x)$ of degree $k$ ($k$\ integer) is given by
\begin{equation}\label{shifJacob1}
J_{L,k}^{\beta,\gamma}(x)=\sum_{i=0}^{k}(-1)^{k-i}\frac{\Gamma(k+\gamma+1)\Gamma(k+i+\beta+\gamma+1)}{\Gamma(i+\gamma+1)\Gamma(k+\beta+\gamma+1)(k-i)!i!L^i}x^i,
\end{equation}
or,
\begin{equation}\label{shifJacob2}
J_{L,k}^{\beta,\gamma}(x)=\sum_{i=0}^{k}\frac{\Gamma(k+\gamma+1)\Gamma(k+i+\beta+\gamma+1)}{\Gamma(i+\beta+1)\Gamma(k+\beta+\gamma+1)(k-i)!i!L^i}(x-L)^i,
\end{equation}
where $$J_{L,k}^{\beta,\gamma}(0)=(-1)^k\frac{\Gamma(k+\gamma+1)}{\Gamma(\gamma+1)k!},$$ and
$$J_{L,k}^{\beta,\gamma}(L)=\frac{\Gamma(k+\beta+1)}{\Gamma(\beta+1)k!}.$$
The orthogonality condition of shifted Jacobi polynomials is
$$ \int_{0}^{L}J_{L,k}^{\beta,\gamma}(x)J_{L,j}^{\beta,\gamma}(x)w_L^{\beta,\gamma}(x)dx=h_k,$$ where
$$w_L^{\beta,\gamma}(x)=x^{\gamma}(L-x)^{\beta} \ \ \text{and}\ h_k=\begin{cases}
\frac{L^{\beta+\gamma+1}\Gamma(k+\beta+1)\Gamma(k+\gamma+1)}{(2k+\beta+\gamma+1)k!\Gamma(k+\beta+\gamma+1)},\ \ k=j,\\
0, \qquad\qquad\qquad\qquad\qquad \ k\neq j.
\end{cases}$$

The expansion of $x^i$ and $(x-L)^i$ in terms of shifted Jacobi polynomials are given, respectively, by:

$$x^i=\frac{(\gamma+1)_i}{(\beta+\gamma+2)_i}\sum_{k=0}^{i}\frac{(-1)^kL^i(-i)_k(\beta+\gamma+2k+1)(\beta+\gamma+2)_{k-1}}{(1+\gamma)_k(\beta+\gamma+i+2)_k}J_{L,k}^{\beta,\gamma}(x),$$

$$(x-L)^i=\frac{(\beta+1)_i}{(\beta+\gamma+2)_i}\sum_{k=0}^{i}\frac{L^i(-i)_k(\beta+\gamma+2k+1)(\beta+\gamma+2)_{k-1}}{(1+\beta)_k(\beta+\gamma+i+2)_k}J_{L,k}^{\beta,\gamma}(x),$$
where $(.)_k$ is Pochhammer's symbol.

Assume $f(t)\in L^2_{w_L^{\beta,\gamma}(x)}(0,L)$, then it can be expanded by means of the shifted Jacobi polynomials as the following form \cite{E. Godoy}:
\begin{equation}
f(x)=\sum_{j=0}^{\infty}c_jJ_{L,j}^{\beta,\gamma}(x),
\end{equation}
where  $$c_j=\frac{1}{h_k}\int_{0}^{L}w_L^{\beta,\gamma}(x)f(x)J_{L,j}^{\beta,\gamma}(x) dx,\ \ \ \ j=0,1,2,\cdots.$$
If we approximate $f(x)$ by the first $M$ term, then we can write
\begin{equation}
f(x)=\sum_{j=0}^{M}c_jJ_{L,j}^{\beta,\gamma}(x).
\end{equation}
We mention here that Chebyshev, Legendre, and ultraspherical polynomials are particular cases of the Jacobi polynomials.

\section{Riesz-Feller fractional derivative of shifted Jacobi polynomials}

For $1<\alpha<2$, depending on definition of Riemann-Liouville fractional derivatives on $[0,L]$,  
\begin{equation} \label{left-deri-x}
_{0}D^{\alpha}_x(x)^k=\frac{\Gamma(k+1)}{\Gamma(k-p+1)}(x)^{k-\alpha},\ \ \ k>-1,
\end{equation}
\begin{equation} \label{left-deri-x-L}
_{x}D^{\alpha}_L(x-L)^k=\frac{(-1)^k\Gamma(k+1)}{\Gamma(k-\alpha+1)}(L-x)^{k-\alpha},\ \ \ k>-1,
\end{equation}

\begin{thm}\label{th-der jacobi left}
The analytic form of the left-side Riemann-Liouville fractional derivative of the shifted Jacobi polynomial on $[0,L]$ is given by:
\begin{eqnarray}
_{0}D_{x}^{\alpha}J_{L,j}^{\beta,\gamma}(x)=\sum_{k=0}^{j}\sum_{i=0}^{k}\Theta_{i,j,k}^{\alpha,\beta,\gamma}\times\Upsilon_{i,k}^{\alpha,\beta,\gamma}\times x^{-\alpha}\times J_{L,j}^{\beta,\gamma}(x),
\end{eqnarray}
where,
\begin{eqnarray}
\Theta_{i,j,k}^{\alpha,\beta,\gamma}&=&\frac{(-1)^{(i+j+k)}\Gamma(1+\beta+\gamma+j+k)\Gamma(1+\gamma+j)\Gamma(1+k)}{\Gamma(1+\beta+\gamma+j)\Gamma(1+\gamma+k)\Gamma(1+k-\alpha)(j-k)!k!},\nonumber\\
\Upsilon_{i,k}^{\alpha,\beta,\gamma}&=&\frac{(-k)_i(1+\gamma)_k(2+\beta+\gamma)_{i-1}(1+\beta+\gamma+2i)}{(1+\gamma)_i(2+\beta+\gamma)_k(2+k+\beta+\gamma)_i}.
\end{eqnarray}
\end{thm}
\begin{proof} See \cite{Bhrawyzakyjacobi}. 
The proof was driven depending on linearity of Riemann-Liouville fractional operator, relation (\ref{left-deri-x}) and the expansion of $x^k$ in terms of shifted Jacobi polynomials.
\end{proof}

\begin{thm}\label{th-der jacobi right}
	The analytic form of the right-side Riemann-Liouville fractional derivative of the shifted Jacobi polynomial on $[0,L]$ is given by:
		\begin{eqnarray}
		_xD_{L}^{\alpha}J_{L,j}^{\beta,\gamma}(x)=\sum_{k=0}^{j}\sum_{i=0}^{k}\overline{\Theta}_{j,k}^{\alpha,\beta,\gamma}\times\overline{\Upsilon}_{i,k}^{\beta,\gamma}\times (L-x)^{-\alpha}\times J_{L,j}^{\beta,\gamma}(x),
		\end{eqnarray}
		where,
		\begin{eqnarray}
		\overline{\Theta}_{j,k}^{\alpha,\beta,\gamma}&=&\frac{(-1)^{(k)}\Gamma(1+\beta+\gamma+j+k)\Gamma(1+\beta+j)\Gamma(1+k)}{\Gamma(1+\beta+\gamma+j)\Gamma(1+\beta+k)\Gamma(1+k-\alpha)(j-k)!k!},\nonumber\\
		\overline{\Upsilon}_{i,k}^{\beta,\gamma}&=&\frac{(-k)_i(1+\beta)_k(2+\beta+\gamma)_{i-1}(1+\beta+\gamma+2i)}{(1+\beta)_i(2+\beta+\gamma)_k(2+k+\beta+\gamma)_i}.
		\end{eqnarray}
			
		\end{thm}
		\begin{proof}See \cite{Bhrawyzakyjacobi}. 
The proof was driven depending on linearity of Riemann-Liouville fractional operator, relation (\ref{left-deri-x-L}) and the expansion of $(x-L)^k$ in terms of shifted Jacobi polynomials.
		\end{proof}
\begin{thm}\label{th-der jacobi rieszfeller}
	The analytic form of the Riesz-Feller fractional derivative of the shifted Jacobi polynomial on $[0,L]$ is given by:
\end{thm}

	\begin{eqnarray}
	D^{\alpha}_{\theta}J_{L,j}^{\beta,\gamma}(x)
	=&&-(c_+\ _0D^{\alpha}_{x}J_{L,j}^{\beta,\gamma}(x)+c_-\ _{x}D^{\alpha}_LJ_{L,j}^{\beta,\gamma}(x)), \nonumber\\
	=&&-\big[c_+\ \sum_{k=0}^{j}\sum_{i=0}^{k}\Theta_{i,j,k}^{\alpha,\beta,\gamma}\times\Upsilon_{i,k}^{\alpha,\beta,\gamma}\times x^{-\alpha}\times J_{L,j}^{\beta,\gamma}(x)\nonumber\\
	&&+c_-\ \sum_{k=0}^{j}\sum_{i=0}^{k}\overline{\Theta}_{j,k}^{\alpha,\beta,\gamma}\times\overline{\Upsilon}_{i,k}^{\beta,\gamma}\times (L-x)^{-\alpha}\times J_{L,j}^{\beta,\gamma}(x)\big], \nonumber\\
		=&&-\sum_{k=0}^{j}\sum_{i=0}^{k}[c_+\ \Theta_{i,j,k}^{\alpha,\beta,\gamma}\times\Upsilon_{i,k}^{\alpha,\beta,\gamma}\times x^{-\alpha}+c_-\ \overline{\Theta}_{j,k}^{\alpha,\beta,\gamma}\times\overline{\Upsilon}_{i,k}^{\beta,\gamma}\times (L-x)^{-\alpha}]J_{L,j}^{\beta,\gamma}(x),\nonumber\\
		=&&\sum_{k=0}^{j}\sum_{i=0}^{k}\Psi_{i,j,k}^{\alpha,\beta,\gamma}\times J_{L,j}^{\beta,\gamma}(x).
	\end{eqnarray}
Where \begin{equation}
\Psi_{i,j,k}^{\alpha,\beta,\gamma}=-[c_+\ \Theta_{i,j,k}^{\alpha,\beta,\gamma}\times\Upsilon_{i,k}^{\alpha,\beta,\gamma}\times x^{-\alpha}+c_-\ \overline{\Theta}_{j,k}^{\alpha,\beta,\gamma}\times\overline{\Upsilon}_{i,k}^{\beta,\gamma}\times (L-x)^{-\alpha}].
\end{equation}

\section{Proceedings of solution for the L\'{e}vy-Feller advection-dispersion equation}
In this section, we will introduce numerical algorithm for approximating the solution of the following L\'{e}vy-Feller advection-dispersion equation:
\begin{equation}\label{levy feller adv sub int}\begin{cases}
\frac{\partial u(x,t)}{\partial t}=dD^{\alpha}_{\theta}u(x,t)-e\frac{\partial u(x,t)}{\partial x}+s(x,t), \ \ \ \ t>0, \ \ 0< x <L, \ \ 1< \alpha \leq2,\\
u(0,t)=0, \qquad \ u(L,t)=0 ,  \ \ \ \ t>0,\\
u(x,0)=f(x), \qquad  \ 0\leqslant x \leqslant L,
\end{cases}
\end{equation}
assuming $u(x,t)=0$ for $x\in \mathbb{R}\backslash [0,L]$.\\
In order to use Jacobi spectral collocation method, we approximate $u(x, t)$ as:
\begin{equation}\label{appr of u}
u_m(x,t)=\sum_{j=0}^{m}u_j(t)J_{L,j}^{\beta,\gamma}(x).
\end{equation}
Depending on properties of Riesz-Feller fractional derivatives we can write
\begin{eqnarray}
D^{\alpha}_{\theta}u_m(x,t)=&&\sum_{j=0}^{m}u_j(t)	D^{\alpha}_{\theta}J_{L,j}^{\beta,\gamma}(x),\nonumber\\
=&& \sum_{j=0}^{m}u_j(t)\sum_{k=0}^{j}\sum_{i=0}^{k}\Psi_{i,j,k}^{\alpha,\beta,\gamma}\times J_{L,j}^{\beta,\gamma}(x), \nonumber\\
=&&\sum_{j=0}^{m}\sum_{k=0}^{j}\sum_{i=0}^{k}u_j(t)\times \Psi_{i,j,k}^{\alpha,\beta,\gamma}\times J_{L,j}^{\beta,\gamma}(x).\nonumber
\end{eqnarray}

So Eq. (\ref{levy feller adv sub int}) take the following form:

\begin{equation}\label{levy adv jacob2}
\sum_{j=0}^{m}\frac{d u_j(t)}{d t}J_{L,j}^{\beta,\gamma}(x)=d\sum_{j=0}^{m}\sum_{k=0}^{j}\sum_{i=0}^{k}u_j(t)\times \Psi_{i,j,k}^{\alpha,\beta,\gamma}\times J_{L,j}^{\beta,\gamma}(x)-e\sum_{j=0}^{m}u_j(t)\frac{d}{dx}J_{L,j}^{\beta,\gamma}(x)+s(x,t).
\end{equation}

 We collocate Eq. (\ref{levy adv jacob2}) at $(m-1)$ points $x_q, \ \ q=1\ ,2\ ,...,\ m-1,\ \ (a<x_q<b)$ as follows:
\begin{equation}\label{levy adv jacob2-2}
\sum_{j=0}^{m}\frac{d u_j(t)}{d t}J_{L,j}^{\beta,\gamma}(x_q)=d\sum_{j=0}^{m}\sum_{k=0}^{j}\sum_{i=0}^{k}u_j(t)\times \Psi_{i,j,k}^{\alpha,\beta,\gamma}\big{|}_{x=x_q}\times J_{L,j}^{\beta,\gamma}(x_q)-e\sum_{j=0}^{m}u_j(t)\frac{d}{dx}J_{L,j}^{\beta,\gamma}(x)\big{|}_{x=x_q}+s(x_q,t).
\end{equation}

Substituting Eq. (\ref{appr of u}) in the initial condition gives us
the constants $u_i$ in the initial case at $t=0$ and substituting
by the same equation in the boundary conditions will give two equations as follows:
\begin{equation}\label{bound-at point}
\sum_{j=0}^{m}(-1)^j\frac{\Gamma(j+\gamma+1)}{\Gamma(\gamma+1)j!}u_j(t)=0,\qquad\sum_{j=0}^{m}\frac{\Gamma(j+\beta+1)}{\Gamma(\beta+1)j!}u_j(t)=0,\qquad
\end{equation}
Equations (\ref{levy adv jacob2-2}) and (\ref{bound-at point}) constitute system of
$(m + 1)$ ordinary differential equations in the unknown $u_j,\ j = 0, 1, ...,m.$ This system will be solved using the trapezoidal rule (which is implicit, second-order and stable method) as described in the following:\\
Let  $0<t_n<T_{final}$ and suppose $\triangle t=\frac{T_{final}}{N} , \ \ t_n=n\triangle t, \ for\  n=0,1,2,...,N,$ then we have the following algebraic system:

\begin{equation}\label{system alg levy adv jac}\begin{cases}
\underset{j=0}{\overset{m}{\ \sum}}	(-1)^{^j}\frac{\Gamma(j+\gamma+1)}{\Gamma(\gamma+1)j!}u_j^n=0,\\

\underset{j=0}{\overset{m}{\ \sum}}\frac{u_j^n-u_j^{n-1}}{\triangle t}J_{L,j}^{\beta,\gamma}(x_q)=\frac{1}{2}\left[\left(d\underset{j=0}{\overset{m}{\sum}}\ \underset{k=0}{\overset{j}{\sum}}\ \underset{i=0}{\overset{k}{\sum}}u_j^n\times \Psi_{i,j,k}^{\alpha,\beta,\gamma}\big{|}_{x=x_q}\times J_{L,j}^{\beta,\gamma}(x_q)\right.\right.\\
\left.\left.\qquad\qquad\qquad\qquad\qquad\qquad\qquad\qquad\qquad\qquad-e\underset{j=0}{\overset{m}{\sum}}u_j^n\frac{d}{dx}J_{L,j}^{\beta,\gamma}(x)\big{|}_{x=x_q}+s_q^n\right)\right. \\ 
\left. \qquad  \qquad \qquad\qquad\qquad \ \ + \left(d\underset{j=0}{\overset{m}{\sum}}\ \underset{k=0}{\overset{j}{\sum}}\ \underset{i=0}{\overset{k}{\sum}}u_j^{n-1}\times \Psi_{i,j,k}^{\alpha,\beta,\gamma}\big{|}_{x=x_q}\times J_{L,j}^{\beta,\gamma}(x_q)\right.\right.\\
\left.\left.\qquad\qquad\qquad\qquad\qquad\qquad\qquad\qquad\qquad\qquad-e\underset{j=0}{\overset{m}{\sum}}u_j^{n-1}\frac{d}{dx}J_{L,j}^{\beta,\gamma}(x)\big{|}_{x=x_q}+s_q^{n-1}\right)	\right] ,\\
\qquad \qquad \qquad \qquad \qquad \qquad \qquad \qquad  q=1,2, ..., m-1, \\
\underset{i=0}{\overset{m}{\ \sum}}\frac{\Gamma(j+\beta+1)}{\Gamma(\beta+1)j!}u_j^n=0,\\
\end{cases}
\end{equation} 


with the initial conditions: $$\sum_{j=0}^{m}u_j^0J_{L,j,q}^{\beta,\gamma}=f_q,\ \ \ q=0,1,2, ..., m,$$
where $u_j^n=u_j(t_n),\quad J_{L,j,q}^{\beta,\gamma}=J_{L,j}^{\beta,\gamma}(x_q),\quad  s_q^n=s(x_q,t_n)\quad and \quad f_q=f(x_q).$\\
System (\ref{system alg levy adv jac}) can be written in a matrix form as the following:
\begin{equation}\label{final system LV adv}
(\boldsymbol{J_1-A})\boldsymbol{U}^n=(\boldsymbol{J_0+A})\boldsymbol{U}^{n-1}+\frac{1}{2}\triangle t(\boldsymbol{S}^n+\boldsymbol{S}^{n-1}),
\end{equation}
 such that:
  $$\boldsymbol{U}^n=(u_0^n,u_1^n,...,u_m^n)^T,$$
  $$\boldsymbol{S}^n=(0,s_1^n,s_2^n...s_{m-1}^n,0)^T,$$

 $$\boldsymbol{A}=\left(
 \begin{array}{cccccc}
  0      & 0   & 0   & 0  & \cdots      & 0  \\
\Omega_{0,1}^{} &\Omega_{1,1}^{} & \Omega_{2,1}^{} & \Omega_{3,1}^{} & \cdots  & \Omega_{m,1}^{} \\
\Omega_{0,2}^{} &\Omega_{1,2}^{} & \Omega_{2,2}^{} & \Omega_{3,2}^{} & \cdots  & \Omega_{m,2}^{} \\
\Omega_{0,3}^{} &\Omega_{1,3}^{} & \Omega_{2,3}^{} & \Omega_{3,3}^{} & \cdots  & \Omega_{m,3}^{} \\
 \vdots & \vdots & \vdots & \vdots & \ddots& \vdots \\
\Omega_{0,m-1}^{} &\Omega_{1,m-1}^{} & \Omega_{2,m-1}^{} & \Omega_{3,m-1}^{} & \cdots  & \Omega_{m,m-1}^{} \\
 0      & 0   & 0   & 0  & \cdots      & 0  \\
 \end{array}
 \right)_{m+1},
 $$

 $$\boldsymbol{J_1}=\left(
 \begin{array}{cccccc}
\frac{\Gamma(\gamma+1)}{\Gamma(\gamma+1)}      & -\frac{\Gamma(1+\gamma+1)}{\Gamma(\gamma+1)1!}   & \frac{\Gamma(2+\gamma+1)}{\Gamma(\gamma+1)2!}   & -\frac{\Gamma(3+\gamma+1)}{\Gamma(\gamma+1)3!}  & \cdots      & (-1)^m\frac{\Gamma(m+\gamma+1)}{\Gamma(\gamma+1)m!} \\
 J_{L,0,1}^{\beta,\gamma} &J_{L,1,1}^{\beta,\gamma} & J_{L,2,1}^{\beta,\gamma} & J_{L,3,1}^{\beta,\gamma} & \cdots  & J_{L,m,1}^{\beta,\gamma} \\
  J_{L,0,2}^{\beta,\gamma} &J_{L,1,2}^{\beta,\gamma} & J_{L,2,2}^{\beta,\gamma} & J_{L,3,2}^{\beta,\gamma} & \cdots  & J_{L,m,2}^{\beta,\gamma} \\
  J_{L,0,3}^{\beta,\gamma} &J_{L,1,3}^{\beta,\gamma} & J_{L,2,3}^{\beta,\gamma} & J_{L,3,3}^{\beta,\gamma} & \cdots  & J_{L,m,3}^{\beta,\gamma}\\
 \vdots & \vdots & \vdots & \vdots & \ddots& \vdots \\
 J_{L,0,m-1}^{\beta,\gamma} &J_{L,1,m-1}^{\beta,\gamma} & J_{L,2,m-1}^{\beta,\gamma} & J_{L,3,m-1}^{\beta,\gamma} & \cdots  & J_{L,m,m-1}^{\beta,\gamma} \\
  \frac{\Gamma(\beta+1)}{\Gamma(\beta+1)}      & \frac{\Gamma(1+\beta+1)}{\Gamma(\beta+1)1!}   & \frac{\Gamma(2+\beta+1)}{\Gamma(\beta+1)2!}   & \frac{\Gamma(3+\beta+1)}{\Gamma(\beta+1)3!}  & \cdots      & \frac{\Gamma(m+\beta+1)}{\Gamma(\beta+1)m!} \\
 \end{array}
 \right)_{m+1},
 $$
 
  $$\boldsymbol{J_0}=\left(
  \begin{array}{cccccc}
  0      & 0   & 0   & 0  & \cdots      & 0 \\
   J_{L,0,1}^{\beta,\gamma} &J_{L,1,1}^{\beta,\gamma} & J_{L,2,1}^{\beta,\gamma} & J_{L,3,1}^{\beta,\gamma} & \cdots  & J_{L,m,1}^{\beta,\gamma} \\
  J_{L,0,2}^{\beta,\gamma} &J_{L,1,2}^{\beta,\gamma} & J_{L,2,2}^{\beta,\gamma} & J_{L,3,2}^{\beta,\gamma} & \cdots  & J_{L,m,2}^{\beta,\gamma} \\
  J_{L,0,3}^{\beta,\gamma} &J_{L,1,3}^{\beta,\gamma} & J_{L,2,3}^{\beta,\gamma} & J_{L,3,3}^{\beta,\gamma} & \cdots  & J_{L,m,3}^{\beta,\gamma}\\
  \vdots & \vdots & \vdots & \vdots & \ddots& \vdots \\
  J_{L,0,m-1}^{\beta,\gamma} &J_{L,1,m-1}^{\beta,\gamma} & J_{L,2,m-1}^{\beta,\gamma} & J_{L,3,m-1}^{\beta,\gamma} & \cdots  & J_{L,m,m-1}^{\beta,\gamma} \\
  0   & 0   & 0   &0  & \cdots        &0 \\
  \end{array}
  \right)_{m+1},
  $$
 \\
 and
 $$\Omega_{j,q}=\frac{1}{2}(\triangle td\sum_{k=0}^{j}\sum_{i=0}^{k}\times \Psi_{i,j,k}^{\alpha,\beta,\gamma}\big{|}_{x=x_q}\times J_{L,j,q}^{\beta,\gamma}+\triangle t e\frac{d}{dx}J_{L,j}^{\beta,\gamma}(x)\big{|}_{x=x_q}),$$ $$ \ i=0,1,2,...,m,\ \ j=1,2,...,m-1, \ \ n=1,2,...,N.$$
Substitute the computed coefficients $u_j,\ j=0,1,2,...,m,$ and the Jacobi polynomials in Eq. (\ref{appr of u}) give us the approximation solutions $(u)$ of the proposed problem.\\

In this work we use the Jacobi Gauss-Lobatto points which are useful for the stability, convergence
and efficiency of the Jacobi spectral collocation method.


\section{Numerical simulations}
Some test examples are introduced in this section to illustrate the accurcy of the presented method. 
\begin{example}\label{ex1-levy feller-ad}\cite{Q. Liu}	
We consider the following L\'{e}vy-Feller advection-dispersion equation in a bounded
domain:
\begin{equation}
\begin{cases}
\frac{\partial u(x,t)}{\partial t}=dD^{\alpha}_{\theta}u(x,t)-e\frac{\partial u(x,t)}{\partial x}, \ \ \ \ t>0, \ \ 0< x <\pi,\ \ \ 1< \alpha \leq2,\\
u(0,t)=0, \qquad \ u(\pi,t)=0 ,  \ \ \ \ t>0,\\
u(x,0)=sin(x), \qquad  \ 0\leqslant x \leqslant \pi.
\end{cases}
\end{equation}

\end{example}
Let $d=1.5, \ \ e=1,\ \ \alpha=1.7, \ \theta=0.3, \ \ T_{final}=0.3,$ table (\ref{table1LFAd}) lists the numerical results computed by explicit finite difference approximation (EFDA) in \cite{Q. Liu} and the presented scheme in this paper JSC for problem (\ref{ex1-levy feller-ad}).

Let $d=1.5, \ \ e=1,\ \  \alpha=1.7, \ \ T_{final}=0.4 \ \text{and} \ \triangle t= 0.008$, figure (\ref{figure1-LFAd}) shows the obtained numerical results by means of the presented scheme JSC $(m=5)$ for different values of $\theta$, which indicates
the skewness.

When $\alpha=2,\ \ d=1,\ \ e=0,$ the exact solution of example(\ref{ex1-levy feller-ad}) is $u(x,t)=sin(x)e^{-t}$. Let us consider $T_{final}=3$ and $\triangle t= 0.05$. Figure (\ref{figure2-LFAd}) shows the exact solution and the obtained numerical results by means of the presented scheme JSC for example.(\ref{ex1-levy feller-ad}) when $m=7$ and shows that the EFDA \cite{Q. Liu} is divergent.
 
\begin{table}[h!]\fontsize{8}{8}
	\centering
	\caption{Comparison of the numerical results calculated by EFDA \cite{Q. Liu} when $h=\pi/100$  and by the presented scheme JSC when $m=5$, and $\beta=\gamma=0$ for example (\ref{ex1-levy feller-ad}), where $\alpha=1.7, \ \theta=0.3 \ and \ t=0.3.$}
	\label{table1LFAd}
\begin{tabular}{|ll|ll|l|}
	\hline
	(x,0.3)&   & EFDA in \cite{Q. Liu} &          & Present method JSC \\ \hline
	0.3142 &   &\ 0.23041       &           &  \ 0.21208      \\\hline
	0.6283 &   &\ 0.40603       &            &  \ 0.38590    \\\hline
	0.9425 &   &\ 0.54876       &            & \  0.51814     \\\hline
	1.2566 &   &\ 0.64661       &          &  \ 0.60546    \\\hline
	1.5708 &   &\ 0.68848       &            & \  0.64455     \\\hline
	1.8850 &   &\ 0.66770       &            &  \ 0.63208     \\\hline
	2.1991 &   &\ 0.58292       &            &  \ 0.56471     \\\hline
	2.5133 &   &\ 0.43764       &            &  \ 0.43913    \\\hline
	2.8274 &   &\ 0.23952       &            &  \ 0.25200    \\\hline
	\end{tabular}
\end{table}

\begin{figure}[h!]
	\centering
	\includegraphics[width=.7\textwidth]{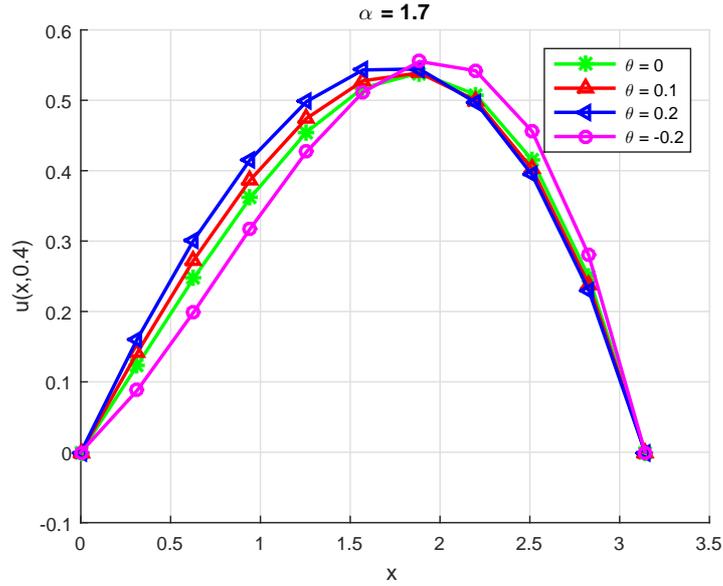}
	\caption{Comparison between, the approximate solution using the proposed method JSC for example (\ref{ex1-levy feller-ad}) at $t = 0.4$ with $\alpha = 1.7$ and different values of $\theta$.}\label{figure1-LFAd}
\end{figure}

\begin{figure}[h!]
	\centering
	\includegraphics[width=0.99\textwidth]{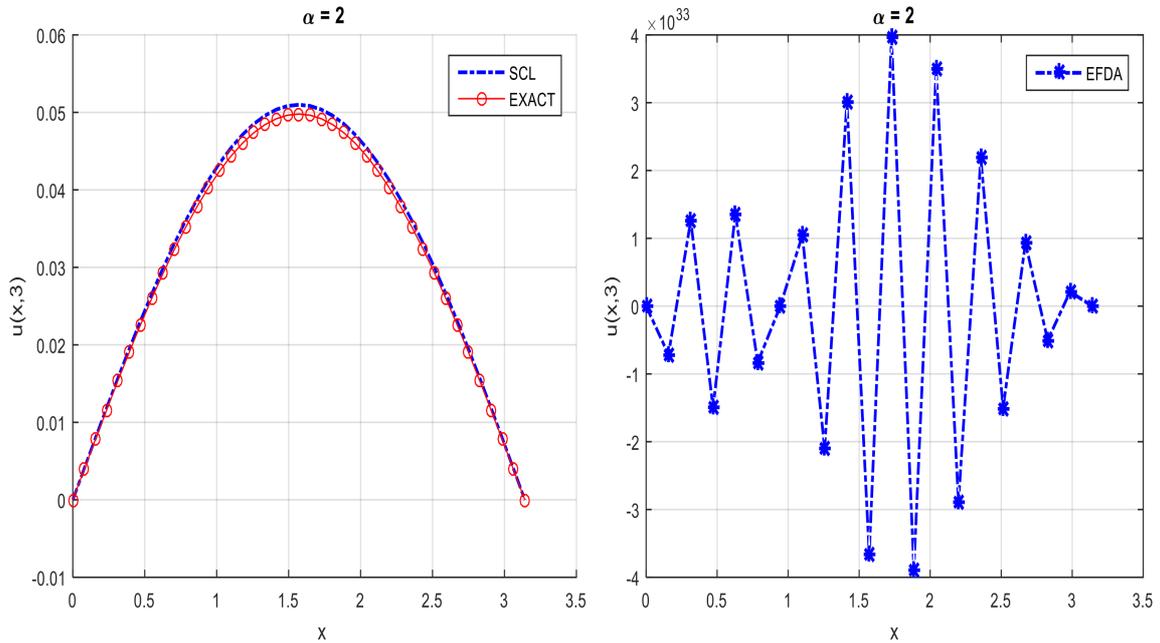}
	\caption{Shows the exact and  numerical solution using EFDA \cite{Q. Liu} and the approximate solution using the proposed method JSC for example (\ref{ex1-levy feller-ad}) at $t = 3$ with $\alpha = 2$.}\label{figure2-LFAd}
\end{figure}
\begin{example}\label{ex2-levy feller-ad}
In this example we consider the following L\'{e}vy-Feller advection-dispersion equation:
	\begin{equation}
	\begin{cases}
	\frac{\partial u(x,t)}{\partial t}=dD^{\alpha}_{\theta}u(x,t)-e\frac{\partial u(x,t)}{\partial x}+s(x,t), \ \ \ \ t>0, \ \ 0< x <1,\\
	u(0,t)=0, \qquad \ u(1,t)=0 ,  \ \ \ \ t>0,\\
	u(x,0)=x(1-x), \qquad  \ 0\leqslant x \leqslant 1,
	\end{cases}
	\end{equation}
	
where $$d=\Gamma(3-\alpha), \qquad e=1,$$ $$ $$ $$s(x,t)=\big{\{}\frac{(2-\alpha)}{sin(\alpha\pi)}\big{(}sin(\frac{(\alpha-\theta)\pi}{2})x^{1-\alpha}+sin(\frac{(\alpha+\theta)\pi}{2})(1-x)^{1-\alpha}\big{)}\quad\quad\quad\qquad\quad$$$$ \qquad\qquad\qquad\ -\frac{2}{sin(\alpha\pi)}\big{(}sin(\frac{(\alpha-\theta)\pi}{2})x^{2-\alpha}+sin(\frac{(\alpha+\theta)\pi}{2})(1-x)^{2-\alpha}\big{)}+\frac{3}{2}x^2-\frac{7}{2}x+1\big{\}}e^{-\frac{3}{2}t},$$ and the exact solution for this equation, when $1<\alpha \leq 2,$ is: $$u(x,t)=x(1-x)e^{-\frac{3}{2}t}.$$
\end{example}

The wighted difference scheme (WDS) which was presented in \cite{M.Ci42} is used here to study numerically example 2, with wight factor $\sigma=0,\ 0.5, \ 1,$ and  the presented scheme JSC also is used. The difference between the approximate solution and the exact solution (absolute error) is given by: $$E_1(x,t)=|u_{_{exact}}(x,t)-u_{_{JSC}}(x,t)| ,\ \ E_2(x,t)=|u_{_{exact}}(x,t)-u_{_{WDS}}(x,t)|,$$ where $E_1\ and \ E_2$ are the errors of the presented scheme JSC and of the WDS \cite{M.Ci42} respectively. Moreover, the maximum absolute errors are given by
$$M_1=max\{E_1(x,t):\ a\leq x\leq b,\ \ 0\leq t\leq T_{final}\},$$ $$\ M_2=max\{E_2(x,t):\ a\leq x\leq b,\ \ 0\leq t\leq T_{final}\},$$ where $M_1 ,\ M_2$ are the maximum absolute errors of JSC and WDS \cite{M.Ci42}, respectively.

In order to show that JSC is more accurate than WDS \cite{M.Ci42}, in table (\ref{table3LFAd}), for $T_{final}=0.5 \ and\  \triangle t= 0.01,$ for problem (\ref{ex2-levy feller-ad}), we compare the errors $E_1 (x,0.5)$ with $E_2(x,0.5)$ when $\alpha=1.4, \ \theta=-0.5$ for various values of $m \  and\  h$.

In table (\ref{table2LFAd}), for $T_{final}=1 \ and\  \triangle t= 0.005,$ for problem (\ref{ex2-levy feller-ad}), we compare the maximum errors $M_1$ when $m=3$ with $M_2$ when $h=0.05$ for various values of $\alpha \  and\  \theta$.

Figure (\ref{figure4-LFAd}) shows the exact solution for Ex.(\ref{ex2-levy feller-ad}) when $T_{final}=2$ and the errors of using  JSC ($m=3$) and WDS \cite{M.Ci42} where $h=0.05$, $\ (\sigma=0)$ and $\triangle t= 0.002$.

Also, Taking $T_{final}=5$ and $\triangle t= 0.02$, figure (\ref{figure3-LFAd}) shows the exact solution for Ex.(\ref{ex2-levy feller-ad}) at $t= 5$ and the obtained numerical results by means of the presented scheme JSC  when $m=3$ and by the WDS $\ (\sigma=0)$ \cite{M.Ci42} when $h=0.1$.
 
\begin{table}[]
	\centering
	\caption{Comparison of the errors calculated by JSC $(\beta=\gamma=0.5$) ($E_1$) and by WDS \cite{M.Ci42} ($E_2$) with $\sigma=0.5$ for example (\ref{ex2-levy feller-ad}) at t= $0.5.$}
	\label{table3LFAd}
	\begin{tabular}{|l|l|l|l|l|l|l|}
		\hline
		& \multicolumn{3}{l|}{\quad\qquad\qquad$E_1(x,0.5)$} & \multicolumn{3}{l|}{\quad\qquad\qquad$E_2(x,0.5)$} \\ \hline
		$x$   & $m=3$     & $m=6$     & $m=12$    & $h=0.1$    &  $h=0.05$ &  $h=0.025$       \\ \hline
		$0$   &  0        &  0        &  0        &  0         & 0         &   0      \\ \hline
		$0.1$ & 1.9827e-05& 1.5831e-06& 1.2334e-06&  1.0946e-02& 9.5111e-03& 8.3772e-03\\ \hline
		$0.2$ & 2.9278e-05& 6.1593e-06& 1.5369e-06&  5.2330e-02& 1.4747e-02&  1.4856e-02\\ \hline
		$0.3$ & 3.0590e-05& 9.4183e-06& 9.5057e-07&  2.1783e-02& 1.7541e-02& 1.9247e-02\\ \hline
		$0.4$ & 2.6003e-05& 1.0927e-05& 6.1780e-06&  5.7803e-02& 1.7939e-02&  2.1571e-02\\ \hline
		$0.5$ & 1.7757e-05& 1.1467e-05& 8.1340e-06&  2.7705e-02& 1.6039e-02& 2.1855e-02\\ \hline
		$0.6$ & 8.0902e-06& 3.1435e-06& 1.2858e-06&  3.8135e-02& 1.2109e-02& 2.0173e-02\\ \hline
		$0.7$ & 7.5812e-06& 9.0309e-07& 9.0856e-07&  1.3635e-02& 6.6847e-03& 1.6691e-02\\ \hline
		$0.8$ & 6.5487e-06& 2.1818e-06& 8.6676e-07&  9.1295e-03& 7.4279e-04&  1.1747e-02\\ \hline
		$0.9$ & 7.0424e-06& 2.3546e-06& 1.7558e-06&  3.5132e-03& 3.7610e-03&  6.0224e-03\\ \hline
		$1$   & 0         &  0        & 0         &   0        &   0       & 0        \\ \hline
	\end{tabular}
\end{table}

\begin{table}[h!]
	\centering
	\caption{Comparison of the maximum errors calculated by JSC $(\beta=\gamma=0)$ ($M_1$) and by WDS \cite{M.Ci42} ($M_2$) for example (\ref{ex2-levy feller-ad})}
	\label{table2LFAd}
	\begin{tabular}{|l|l|l|l|l|}
		\hline
		&\qquad $M_1$ & \multicolumn{3}{l|}{\qquad\qquad\qquad $M_2$} \\ \hline
   $\quad \alpha,\ \ \qquad \theta$& \quad m=3&  $\sigma=1$     & $\sigma=0.5$      & $\sigma=0$      \\ \hline
$\alpha=1.8,\ \theta=0.1$	& 1.0995e-05 &  divergent     & 1.0064e-02     & 9.4012e-03       \\ \hline
$\alpha=1.6,\ \theta=0.1$	& 4.5779e-05 & 2.0452e-02      & 2.0031e-02      & 1.9644e-02      \\ \hline
$\alpha=1.6,\ \theta=0.3$   & 8.9812e-05 & 2.2966e-02      &  2.2429e-02       & 2.1961e-02      \\ \hline
$\alpha=1.4,\ \theta=0.3$   & 7.6254e-05 & 3.9459e-02      & 3.8952e-02      & 1.2893e-02      \\ \hline
$\alpha=1.4,\ \theta=0.5$	& 6.4506e-05 & 4.4378e-02      & 4.3804e-02      & 4.3232e-02      \\ \hline
$\alpha=1.2,\ \theta=0.5 $  & 5.2149e-05 & 6.3144e-02      & 6.2514e-02      & 6.1884e-02      \\ \hline
$\alpha=1.2,\ \theta=-0.5$	& 1.3202e-05 & divergent      & divergent      & divergent      \\ \hline
$\alpha=1.1,\ \theta=-0.5$	& 4.1385e-05 & divergent      & divergent      &  divergent     \\ \hline
	\end{tabular}
\end{table}

\begin{figure}[h!]
	\centering
	\includegraphics[width=1.05\textwidth]{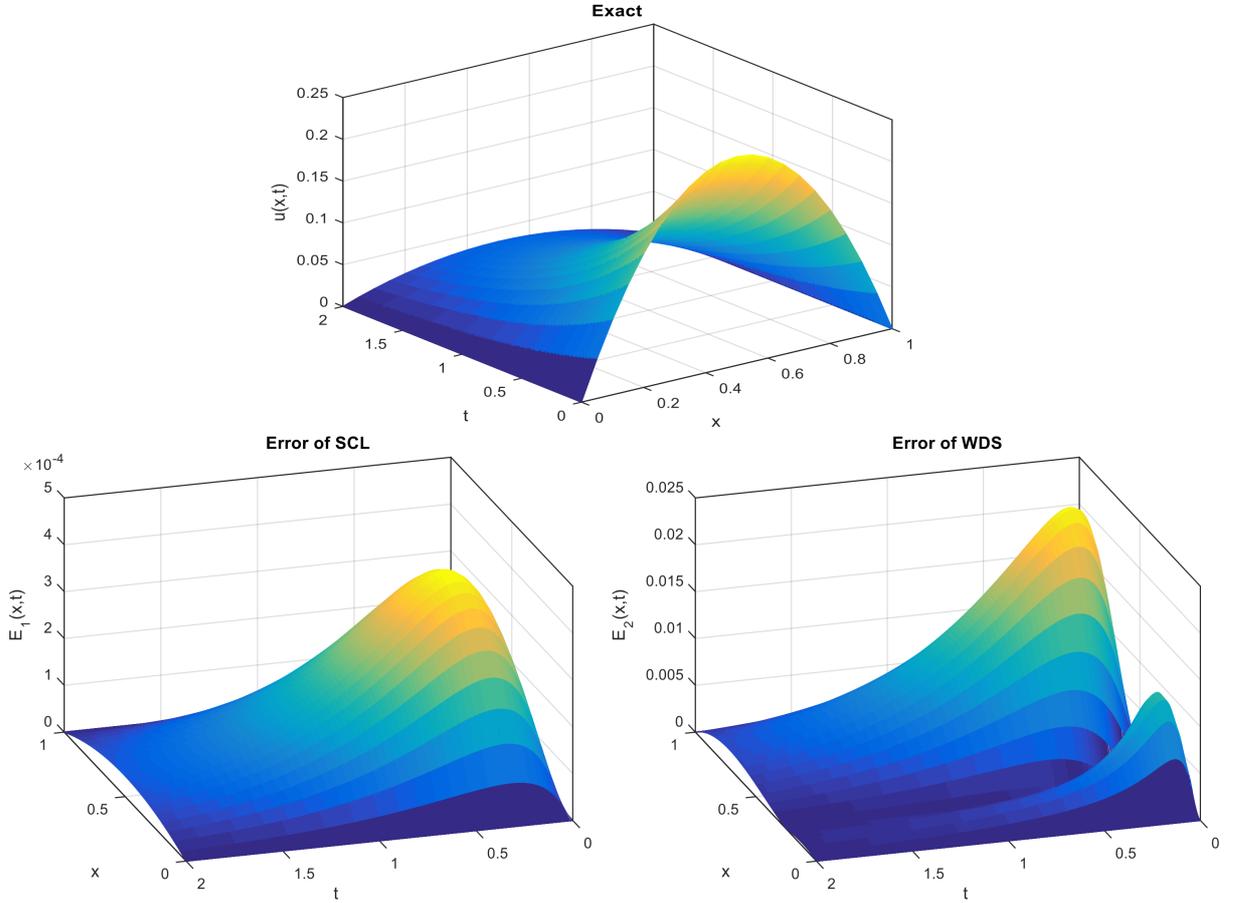}
	\caption{Comparison between, the errors obtain by JSC $(\beta=\gamma=-0.5)$ and by WDS \cite{M.Ci42} for example (\ref{ex2-levy feller-ad}) when $T_{final} = 2$, $\alpha =1.7,$\ and $\theta=0.2$.}\label{figure4-LFAd}
\end{figure}

\begin{figure}[h!]
	\centering
	\includegraphics[width=.71\textwidth]{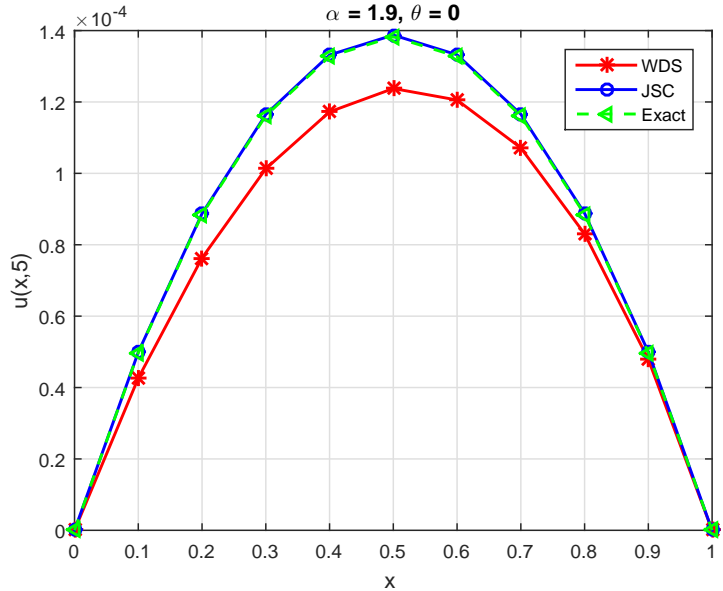}
	\caption{Comparison between, the numerical solution using JSC  and by WDS \cite{M.Ci42} for example (\ref{ex2-levy feller-ad}) at $t = 5$ with $\alpha =1.9,\ \theta=0$.}\label{figure3-LFAd}
\end{figure}



\newpage
\section{Conclusion}
An accurate numerical method is constructed to approximate the solutions of the L\'{e}vy-Feller advection-dispersion equation. This algorithm is based on the Jacobi collocation method in combination with the trapezoidal rule to create a system of algebraic equations of the unknown coefficients of the spectral expansion. One of the main advantages of the presented algorithms is that the availability for application on fractional differential equations, inclusive the case of the fractional derivative is defined in the Riesz-(Feller) sense. Another advantage of the proposed technique is that the high accurate approximate solutions are achieved by using a few number of terms of the suggested expansion. Comparisons between our approximate solutions of the problems with its exact solutions and with the approximate solutions achieved by other methods were introduced to highlight the validity and the accuracy of the proposed scheme. Summarizing, when we used the proposed method for solving both  L\'{e}vy-Feller diffusion  equation and  L\'{e}vy-Feller advection-dispersion equation we found that this method is more efficient than finite difference method which was used in \cite{H. Zhang} and \cite{Q. Liu} for solving these two equations respectively.
%
\bibliographystyle{alpha}\fontsize{10}{10}

\end{document}